\documentclass{article}
\usepackage{amsfonts}
\usepackage{amsmath}

\newtheorem{theorem}{Theorem}[section]

\newtheorem{definition}[theorem]{Definition}

\newtheorem{lemma}[theorem]{Lemma}

\newtheorem{remark}[theorem]{Remark}

\newenvironment{proof}[1][Proof]{\noindent\textbf{#1.} }{\ \rule{0.5em}{0.5em}}
\input{tcilatex}
\begin{document}

\title{Strong Local Nondeterminism of Spherical Fractional Brownian Motion}
\author{Xiaohong Lan \thanks{%
Corresponding author. Research of X. Lan is supported by NSFC grants
11501538. E-mail: xhlan@ustc.edu.cn} \\
School of Mathematical Sciences \\
University of Science and Technology of China \and Yimin Xiao \thanks{%
Research of Y. Xiao is partially supported by NSF grants DMS-1612885 and
DMS-1607089. E-mail: xiao@stt.msu.edu} \\
Department of Statistics and Probability\\
Michigan State University}
\maketitle

\begin{abstract}
Let $B = \left\{ B\left( x\right),\, x\in \mathbb{S}^{2}\right\} $ be the
fractional Brownian motion indexed by the unit sphere $\mathbb{S}^{2}$ with
index $0<H\leq \frac{1}{2}$, introduced by Istas \cite{IstasECP05}. We
establish optimal upper and lower bounds for its angular power spectrum $%
\{d_\ell, \ell = 0, 1, 2, \ldots\}$, and then exploit its high-frequency
behavior to establish the property of its strong local nondeterminism of $B$.
\end{abstract}



\section{Introduction}

The spherical fractional Brownian motion (SFBM, for brevity) was introduced
by Istas in 2005 \cite{IstasECP05}, as an extension of the spherical
Brownian motion of L\'evy \cite{Levy} as well as a spherical analogue of
fractional Brownian motion indexed by the Euclidean spaces. Later Istas \cite%
{IstasSPL06,Istas07} established the Karhunen-Lo\`{e}ve expansion and
studied quadratic variations of the spherical fractional Brownian motion.

The purpose of this paper is to investigate the property of strong local
nondeterminism (SLND) for the spherical fractional Brownian motion. This is
motivated by studies of sample path properties of Gaussian random fields
indexed by the Euclidean space $\mathbb{R}^N$ and by the currently
increasing interest in stochastic modeling of spherical data in statistics,
cosmology and other applied areas (see below).

The concept of local nondeterminism (LND) of a Gaussian process was first
introduced by Berman \cite{Berman73} to unify and extend his methods for
studying the existence and joint continuity of local times of real-valued
Gaussian processes. Roughly speaking, a Gaussian process is said to have the
LND property if has locally approximately independent increments, see  \cite[%
Lemma 2.3]{Berman73} for precise description. This property allowed Berman
to overcome some difficulties caused by the complex dependence structure of
a non-Markovian Gaussian process for studying its local times. Pitt \cite%
{Pitt78} and Cuzick \cite{Cuzick82} extended Berman's LND to Gaussian random
fields. However, the property of LND is not enough for establishing fine
regularity properties such as the law of the iterated logarithm and the
uniform modulus of continuity for the local times or self-intersection local
times of Gaussian random fields. For studying these and many other problems
on Gaussian random fields, the appropriate properties of strong local
nondeterminism (SLND) have proven to be more powerful. Instead of recalling
definitions of various forms of strong local nondeterminism for (isotropic
or anisotropic) Gaussian random fields indexed by $\mathbb{R}^N$ and their
applications, we refer to Xiao \cite{X07, X1, X2} for more information.

Recently, Lan, Marinucci and Xiao \cite{LanMarXiao1} have studied the SLND
property of a class of Gaussian random fields indexed by the unit sphere $%
\mathbb{S}^{2}$, which are also called spherical Gaussian random fields. The
main difference between \cite{LanMarXiao1} and the aforementioned work for
Gaussian fields indexed by the Euclidean space is that [16] takes the
spherical geometry of $\mathbb{S}^{2}$ into full consideration and its
method relies on harmonic analysis on the sphere. More specifically, Lan,
Marinucci and Xiao have considered a centered isotropic Gaussian random
field $T = \{T(x), x \in \mathbb{S}^{2}\}$. That is, $T$ satisfies 
\begin{equation}  \label{Eq:Iso}
\mathbb{E} \big(T(x) T(y)\big) = {\mathbb{E}} \big(T(gx) T(gy)\big)
\end{equation}
for all $g\in SO(3)$ which is the group of rotations in $\mathbb{R}^3$. See 
\cite{MarPecbook} for a systematic account on random fields on $\mathbb{S}%
^{2}$. By applying harmonic analytic tools on the sphere, Lan, Marinucci and
Xiao \cite{LanMarXiao1} have proved that the SLND property of an isotropic
Gaussian field $T$ on $\mathbb{S}^{2}$ is determined by the high-frequency
behavior of its angular power spectrum. Moreover, by applying SLND, they
have established exact uniform modulus of continuity for a class of
isotropic Gaussian fields on $\mathbb{S}^{2}$.

Since SFBM $B = \left\{ B\left( x\right),\, x\in \mathbb{S}^{2}\right\} $ is
not isotropic in the sense of \eqref{Eq:Iso}, the results on SLND in \cite%
{LanMarXiao1} are not directly applicable. In our approach we will make use
of the Karhunen-Lo\`{e}ve expansion of the spherical fractional Brownian
motion obtained by Istas \cite{IstasSPL06} and derive optimal upper and
lower bounds for the coefficients $\{d_\ell, \ell = 0, 1, \ldots\}$ (see %
\eqref{Eq:dl} below for the definition). These bounds for $\{d_\ell\}$
correct the last part of Theorem 1 in \cite{IstasSPL06} and will be useful
for studying the dependence structures and sample path properties of SFBM.
This paper provides an important step towards this direction. More
specifically, we demonstrate that the coefficients $\{d_\ell\}$ play the
same role as the angular power spectrum of an isotropic Gaussian field on $%
\mathbb{S}^{2}$ in \cite{LanMarXiao1} and their high frequency behavior
determines the property of strong local nondeterminism of SFBM. For this
reason, we will also call the sequence $\{d_\ell, \ell = 0, 1, \ldots\}$ the
angular power spectrum of SFBM. 

Similarly to the cases of Gaussian random fields indexed by the Euclidean
space $\mathbb{R}^N$ (cf. \cite{X07, X1,X2}), we expect that the SLND
property in Theorem \ref{Th:SLND} will be useful for studying regularity
(e.g., the exact modulus of continuity, exact modulus of
nondifferentiability, etc) and fractal properties of SFBM. This will be
carried out in a subsequent paper \cite{LX17}.

Our analysis on SFBM and other spherical Gaussian random fields is strongly
motivated by applications in a number of scientific areas, such as
geophysics, astrophysics, cosmology, and atmospheric sciences (see e.g. \cite%
{AlMi,CaMa, deBoy,Dode2004}). Huge data sets from satellite missions such as
the Wilkinson Microwave Anisotropy Probe (\emph{WMAP}) of NASA (see
http://map.gsfc.nasa.gov/) and the \emph{Planck} mission of the European
Space Agency (see http://sci.esa.int/planck/53103-planck-cosmology/) have
been collected and made publicly available. Spherical random fields (usually
assumed to be Gaussian) have been proposed for modeling such data sets.

Related to aforementioned aspects, we also mention that in probability and
statistics literature various isotropic or anisotropic Gaussian random
fields on $\mathbb{S}^{2}$ have been constructed and studied (see e.g. \cite%
{EsIstas10, HZR1, HZR2,JS,PanTala}). Excursion probabilities and topological
properties of excursion sets of isotropic Gaussian random fields on $\mathbb{%
S}^{2}$ have been studied in \cite{CX16,MaVa}. Many interesting questions on
probabilistic and statistical properties of anisotropic Gaussian random
fields on the sphere can be raised. In order to study these problems. it
would be interesting to establish appropriate properties of strong local
nondeterminism for anisotropic Gaussian random fields on $\mathbb{S}^{2}$.

The rest of the paper is organized as follows. In Section 2, we recall
briefly some background material on SFBM, including its Karhunen-Lo\`{e}ve
expansion from \cite{IstasSPL06}, and an analysis of its random
coefficients. Our main result in this section is Theorem \ref{APSbound},
which provides optimal upper and lower bounds for the angular power spectrum 
$\{d_\ell, \ell = 0, 1, \ldots\}$ of SFBM. In Section 3, we combine the high
frequency behavior of $\{d_\ell\}$ and Proposition 7 in \cite{LanMarXiao1}
to establish the property of strong local nondeterminism for SFBM. 


\section{SFBM and asymptotic behavior of its angular power spectrum}

Let $N$ be the North pole on $\mathbb{S}^{2},$ and $d_{\mathbb{S}^{2}}\left(
x,y\right) $ the geodesic distance between $x$ and $y$ on $\mathbb{S}^{2}$.
Recall from Istas \cite{IstasECP05} the definition of SFBM.

\begin{definition}
\label{def} The SFBM $B= \left\{ B\left( x\right),\, x\in \mathbb{S}%
^{2}\right\} $ is a centered Gaussian random field with 
\begin{equation}
B\left( N\right) =0,\ a.s.  \label{North}
\end{equation}%
and%
\begin{equation}
\mathbb{E}\left[ B\left( x\right) -B\left( y\right) \right] ^{2}= d_{\mathbb{%
S}^{2}}\left( x,y\right) ^{2H},  \label{Variogram}
\end{equation}%
for any $x,y\in \mathbb{S}^{2}$ and $0<H\leq \frac{1}{2}.$
\end{definition}

It is well known that fractional Brownian motion indexed by $\mathbb{R}^{d}$
can be defined for every $0< H \le 1$. This is different from the case when
the index set is $\mathbb{S}^{2}$. Istas \cite{IstasECP05} proved that SFBM 
exists if and only if the Hurst index $H\in (0,\frac{1}{2}].$ When $H=\frac{1%
}{2}$, it is the classical L\'{e}vy spherical Brownian motion, see \cite%
{Levy, Noda}.

Our work is based on the following Karhunen-Lo\`{e}ve expansion of $\big\{%
B\left( x\right) ,$ $x\in \mathbb{S}^{2}\big\}$ proved by Istas \cite[%
Theorem 1]{IstasSPL06}: 
\begin{equation}
B\left( x\right) =\dsum\limits_{\ell =0}^{\infty }\dsum\limits_{m=-\ell
}^{\ell }i\sqrt{\pi d_{\ell }}\,\varepsilon _{\ell m}\left( Y_{\ell m}\left(
x\right) -Y_{\ell m}\left( N\right) \right) .  \label{repBM}
\end{equation}%
Here $\left\{ Y_{\ell m}:\ell =0,1,...,m=-\ell ,...,\ell \right\} $ are the
spherical harmonic functions on $\mathbb{S}^{2}$; that is, they are
eigenfunctions of the spherical Laplacian: 
\begin{equation*}
\Delta _{{\mathbb{S}}^{2}}=\frac{1}{\sin \theta }\frac{\partial }{\partial
\theta }\left\{ \sin \theta \frac{\partial }{\partial \theta }\right\} +%
\frac{1}{\sin ^{2}\theta }\frac{\partial ^{2}}{\partial \phi ^{2}},
\end{equation*}%
where $(\theta ,\phi )\in \left[ 0,\pi \right] \times \left[ 0,2\pi \right) $
denotes the spherical coordinates of $x\in \mathbb{S}^{2}$. More precisely, $%
\left\{ Y_{\ell m}:\ell =0,1,...,m=-\ell ,...,\ell \right\} $ satisfy 
\begin{equation*}
\Delta _{\mathbb{S}^{2}}Y_{\ell m}(\theta ,\phi )=-\ell (\ell +1)Y_{\ell
m}(\theta ,\phi ).
\end{equation*}%
It is known (cf. \cite{MarPecbook}) that $\left\{ Y_{\ell m}:\ell
=0,1,...,m=-\ell ,...,\ell \right\} $ form an orthonormal basis for the
space $L^{2}\left( \mathbb{S}^{2},d\sigma \right) $, where $d\sigma =\sin
\theta d\theta d\phi $ is the Lebesgue measure on $\mathbb{S}^{2}$. 
An explicit form for spherical harmonics is given by 
\begin{equation}
\begin{split}
Y_{\ell m}\left( \theta ,\phi \right) & =\sqrt{\frac{2\ell +1}{4\pi }\frac{%
\left( \ell -m\right) !}{\left( \ell +m\right) !}}P_{\ell m}\left( \cos
\theta \right) e^{im\phi },\ \ \hbox{ for }m\geq 0, \\
Y_{\ell m}\left( \theta ,\phi \right) & =\left( -1\right) ^{m}\overline{Y}%
_{\ell ,-m}\left( \theta ,\phi \right) ,\qquad \qquad \qquad \quad 
\hbox{
for }m<0,
\end{split}
\label{Eq:YlmConj}
\end{equation}%
where $\overline{z}$ denotes the complex conjugate of $z\in \mathbb{C},$ and 
$P_{\ell m}\left( \cos \vartheta \right) $ are the associated Legendre
functions (cf. \cite[pp.315--316]{MarPecbook}) defined in terms of the
Legendre polynomials $\{P_{\ell },\ell =0,1,\ldots \}$ as 
\begin{equation*}
P_{\ell m}\left( x\right) =(-1)^{m}(1-x^{2})^{m/2}\frac{d^{m}}{dx^{m}}%
P_{\ell }\left( x\right) 
\end{equation*}%
for $m=0,1,\ldots ,\ell $, and 
\begin{equation*}
P_{\ell m}\left( x\right) =(-1)^{m}\frac{(\ell -m)!}{(\ell +m)!}\,P_{\ell
,-m}\left( x\right) 
\end{equation*}%
for $m$ negative. 
Moreover, the following orthonormality property holds:%
\begin{equation}
\int_{\mathbb{S}^{2}}Y_{lm}\left( x\right) \overline{Y}_{l^{\prime
}m^{\prime }}\left( x\right) d\sigma \left( x\right) =\delta _{l}^{l^{\prime
}}\delta _{m}^{m^{\prime }}.  \label{Eq:YlmOrth}
\end{equation}

Now recall from \cite{IstasSPL06} that, for all $\ell \geq 0$, the
coefficients $d_{\ell }$ in (\ref{repBM}) are defined by 
\begin{equation}
d_{\ell }=\frac{1}{2\pi }\int_{\mathbb{S}^{2}}d_{\mathbb{S}^{2}}\left(
x,N\right) ^{2H}P_{\ell }\left( \left\langle x,N\right\rangle \right)
d\sigma \left( x\right) ,  \label{Eq:dl}
\end{equation}%
where $\left\langle \cdot ,\cdot \right\rangle $ is the usual inner product
in $\mathbb{R}^{3}$. Similarly to the angular power spectrum of an isotropic
Gaussian random field on $\mathbb{S}^{2}$, the sequence $\left\{ d_{\ell
},\ell =0,1,2,...\right\} $ plays an important role in determining the
dependence structure and other probabilistic properties of the SFBM $%
\{B\left( x\right) ,x\in \mathbb{S}^{2}\}.$ Hence we will also refer to $%
\left\{ d_{\ell },\ell =0,1,2,...\right\} $ as the \emph{angular power
spectrum} of SFBM.

Moreover, it follows from Theorem 1 of Istas \cite{IstasSPL06} that the
Karhunen-Lo\`{e}ve expansion (\ref{repBM}) holds in $L^{2}\left( \mathbb{S}%
^{2}\times \Omega ,d\sigma \otimes \mathbb{P}\right) $ sense and in $%
L^{2}\left( \mathbb{P}\right) $ sense for every fixed $x\in \mathbb{S}^{2}$.
Hence one can represent the random coefficients $\varepsilon _{\ell m},\ell
=0,1,...,m=-\ell ,...,\ell $ in (\ref{repBM}) by 
\begin{equation*}
\varepsilon _{\ell m}=-i\left( \pi d_{\ell }\right) ^{-1/2}\int_{\mathbb{S}%
^{2}}B(x)\overline{Y_{\ell m}}(x)d\sigma (x),
\end{equation*}%
where the equality holds in $L^{2}(\mathbb{P})$ sense. 
Notice that $\{\varepsilon _{\ell m}\}_{lm}$ is a set of complex-valued
Gaussian random variables. Obviously, $\mathbb{E}\left( \varepsilon _{\ell
m}\right) =0$ for all $\ell ,m,$ due to the zero-mean property of $B(x).$
Moreover, recall (\ref{Eq:YlmConj}) and (\ref{Eq:YlmOrth}), we can verify
that 
\begin{equation}
\varepsilon _{\ell m}=\left( -i\right) ^{m}\varepsilon _{\ell ,-m}\ 
\hbox{
and }\ \mathbb{E}\left( \varepsilon _{\ell _{1}m_{2}}\overline{\varepsilon }%
_{\ell _{2}m_{2}}\right) =\delta _{\ell _{2}}^{\ell _{1}}\delta
_{m_{2}}^{m_{1}}.  \label{indep-coef}
\end{equation}%
Therefore, $\left\{ \varepsilon _{\ell m},\ell \geq 0,m=0,1,...,\ell
\right\} $ is a set of $i.i.d.$ standard complex Gaussian random variables.


The following is the main result of this section. The bounds in (\ref%
{Bound-dl}) correct the last part of Theorem 1 in \cite{IstasSPL06}. The
high-frequency behavior of $d_{\ell }$ in \eqref{Bound-dl} is essential for
proving the SLND property of $B$ in Theorem \ref{Th:SLND} below. Together,
they will allow us to study precise analytic and geometric properties of the
sample functions of SFBM. See \cite{LX17} for further information.

\begin{theorem}
\label{APSbound} Let $\left\{ B(x),x\in \mathbb{S}^{2}\right\} $ be the
spherical fractional Brownian motion of index $H\in (0,1/2]$. There exists a
uniform constant $K_{1}$ 
such that 
\begin{equation}
\begin{split}
& K_{1}^{-1}\leq d_{0}\leq K_{1} \\
K_{1}^{-1}\ell ^{-(2H+2)}& \leq d_{\ell }\leq K_{1}\ell ^{-(2H+2)}\quad %
\hbox{ for all }\,\ell =1,2,....
\end{split}
\label{Bound-dl}
\end{equation}
\end{theorem}

\begin{proof}
We work in spherical coordinates $\left( \theta ,\phi \right) $ with $\theta
\in \left[ 0,\pi \right] $ and $\phi \in \left[ 0,2\pi \right] .$ By the
definition of $d_{\ell }$ in (\ref{Eq:dl}) and a change of variable $x=\cos
\theta $ we obtain 
\begin{equation}
d_{\ell }=\int_{0}^{\pi }\theta ^{2H}P_{\ell }\left( \cos \theta \right)
\sin \theta d\theta .  \label{Eq:dl1}
\end{equation}%
Recall from \cite{IstasSPL06} the Dirichlet-Mehler representation for $%
P_{\ell }\left( \cos \theta \right) (0<\theta <\pi )$, 
\begin{equation*}
P_{\ell }\left( \cos \theta \right) =\frac{\sqrt{2}}{\pi }\int_{\theta
}^{\pi }\frac{\sin \left( (\ell +\frac{1}{2})\varphi \right) }{\sqrt{\cos
\theta -\cos \varphi }}d\varphi ,
\end{equation*}%
we see that (\ref{Eq:dl1}) becomes 
\begin{equation*}
\begin{split}
d_{\ell }& =\frac{\sqrt{2}}{\pi }\int_{0}^{\pi }\theta ^{2H}\sin \theta
\int_{\theta }^{\pi }\frac{\sin \left( \left( \ell +\frac{1}{2}\right)
\varphi \right) }{\sqrt{\cos \theta -\cos \varphi }}d\varphi d\theta  \\
& =\frac{\sqrt{2}}{\pi }\int_{0}^{\pi }\sin \left( \Big(\ell +\frac{1}{2}%
\Big)\varphi \right) \left[ \int_{0}^{\varphi }\frac{\theta ^{2H}\sin \theta 
}{\sqrt{\cos \theta -\cos \varphi }}\,d\theta \right] \,d\varphi .
\end{split}%
\end{equation*}%
From the elementary fact that $\sin \frac{\theta }{2}\leq \frac{\theta }{2}%
\leq \theta \left[ 1-\frac{1}{6}\left( \frac{\theta }{2}\right) ^{2}\right]
\leq 2\sin \frac{\theta }{2}$ for all $0\leq \theta \leq \pi ,$ it follow
that 
\begin{equation*}
\left( \sin \frac{\theta }{2}\right) ^{2H}\leq \left( \frac{\theta }{2}%
\right) ^{2H}\leq 2^{2H}\cdot \left( \sin \frac{\theta }{2}\right)
^{2H},\quad 0\leq \theta \leq \pi .
\end{equation*}%
Hence, we have 
\begin{equation}
\widetilde{d}_{\ell }\leq d_{\ell }\leq 2^{2H}\widetilde{d}_{\ell },
\label{eq:d-dti}
\end{equation}%
where 
\begin{equation*}
\widetilde{d}_{\ell }=\frac{2^{2H+1/2}}{\pi }\int_{0}^{\pi }\sin \left(
\left( \ell +\frac{1}{2}\right) \varphi \right) \int_{0}^{\varphi }\frac{%
\left( \sin \frac{\theta }{2}\right) ^{2H}\sin \theta }{\sqrt{\cos \theta
-\cos \varphi }}\,d\theta d\varphi .
\end{equation*}%
Meanwhile, by a change of variable $u=\sin \frac{\theta }{2}$ in the inside
integral, we have 
\begin{equation*}
\begin{split}
\int_{0}^{\varphi }\frac{\left( \sin \frac{\theta }{2}\right) ^{2H}\sin
\theta }{\sqrt{\cos \theta -\cos \varphi }}d\theta & =2^{\frac{3}{2}%
}\int_{0}^{\sin \frac{\varphi }{2}}\frac{u^{2H+1}}{\sqrt{\sin ^{2}\frac{%
\varphi }{2}-u^{2}}}du \\
& =2^{\frac{1}{2}}\left( \sin \frac{\varphi }{2}\right) ^{2H+1}\int_{0}^{1}%
\frac{v^{H}}{\sqrt{1-v}}dv \\
& =2^{\frac{1}{2}}B\left( H+1,\frac{1}{2}\right) \left( \sin \frac{\varphi }{%
2}\right) ^{2H+1},
\end{split}%
\end{equation*}%
where $B\left( \cdot ,\cdot \right) $ is the Beta function defined on $%
\mathbb{R}
^{+}\times 
\mathbb{R}
^{+}.$ Consequently, 
\begin{equation}
\widetilde{d}_{\ell }=\frac{2^{2H+1}}{\pi }B\left( H+1,\frac{1}{2}\right)
\int_{0}^{\pi }\sin \left( \left( \ell +\frac{1}{2}\right) \varphi \right)
\left( \sin \frac{\varphi }{2}\right) ^{2H+1}d\varphi .  \label{d-tilda}
\end{equation}%
For $\ell =0,$ 
\begin{equation*}
\widetilde{d}_{0}=\frac{2^{2H+1}}{\pi }B\left( H+1,\frac{1}{2}\right)
\int_{0}^{\pi }\left( \sin \frac{\varphi }{2}\right) ^{2H+2}d\varphi ,
\end{equation*}%
which is readily seen that $\left\vert \widetilde{d}_{0}\right\vert \leq 8.$
For $\ell \geq 1,$ we will make use of some techniques from complex
analysis. Let $z=e^{i\varphi /2},$ then the integral in (\ref{d-tilda}) can
be written as 
\begin{equation}
\int_{0}^{\pi }\sin \left( \left( \ell +\frac{1}{2}\right) \varphi \right)
\left( \sin \frac{\varphi }{2}\right) ^{2H+1}d\varphi =2\func{Im}\int_{%
\mathcal{C}}f_{\ell }\left( z\right) dz,  \label{dti-Imintf}
\end{equation}%
where $\mathcal{C}=\left\{ z:z=e^{it},\ 0\leq t\leq \pi /2\right\} $ is the
unit circle in the first quadrant with the direction counterclockwise and $%
f_{\ell }\left( z\right) $ is the principal branch of the complex function%
\begin{equation*}
F_{\ell }\left( z\right) =\frac{z^{2\ell }}{2^{2H}i^{2H+1}}\left( z-\frac{1}{%
z}\right) ^{2H+1},\quad z\in 
\mathbb{C}
,
\end{equation*}%
Obviously, $f_{\ell }(z)$ is analytic in $\mathbb{C}\backslash \left\{ 0,\pm
1\right\} $. Moreover, let $\epsilon >0$ be an arbitary small value, $%
\mathcal{L}_{1}$ and $\mathcal{L}_{2}$ the two line segments defined by 
\begin{equation*}
\mathcal{L}_{1}=\left\{ z:z=i\left( 1-t\right) ,\ \epsilon \leq t\leq
1\right\} 
\end{equation*}%
and 
\begin{equation*}
\mathcal{L}_{2}=\left\{ z:z=t,\epsilon \leq t\leq 1-\epsilon \right\} ;
\end{equation*}%
Meantime, let $\mathcal{C}_{1}$ and $\mathcal{C}_{2}$ be the circles with
radius $\epsilon $ in the first and second quadrants, defined respectively
by 
\begin{equation*}
\mathcal{C}_{1}=\left\{ z:z=\epsilon e^{it},0\leq t\leq \pi /2\right\} 
\end{equation*}%
and 
\begin{equation*}
\mathcal{C}_{2}=\left\{ z:z=\epsilon e^{i\left( \pi -t\right) }+1,0\leq
t\leq \pi /2\right\} .
\end{equation*}%
Now consider the following integrals 
\begin{eqnarray*}
\int_{\mathcal{L}_{1}}f_{\ell }\left( z\right) dz &=&\frac{i^{2\ell -2H-1}}{%
2^{2H}i^{2H+1}}\int_{\epsilon }^{1}\left( 1-t\right) ^{2\ell -2H-1}\left[
-\left( 1-t\right) ^{2}-1\right] ^{2H+1}dt \\
&=&\frac{\left( -1\right) ^{\ell }}{2^{2H}}\int_{\epsilon }^{1}u^{2\ell
-2H-1}\left[ u^{2}+1\right] ^{2H+1}dt,
\end{eqnarray*}%
\begin{eqnarray*}
\int_{\mathcal{L}_{2}}f_{\ell }\left( z\right) dz &=&\frac{1}{2^{2H}i^{2H+1}}%
\int_{\epsilon }^{1}t^{2\ell -2H-1}\left( t^{2}-1\right) ^{2H+1}dt \\
&=&\frac{i^{2H+1}}{2^{2H}}\int_{\epsilon }^{1}\left( 1-v\right) ^{\ell
-H-1/2}v^{2H+1}dt,
\end{eqnarray*}%
and%
\begin{equation*}
\int_{\mathcal{C}_{1}}f_{\ell }\left( z\right) dz=\frac{1}{2^{2H}i^{2H+1}}%
\int_{0}^{\frac{\pi }{2}}\left( \epsilon e^{i\left( \pi -t\right) }\right)
^{2\ell -2H-1}\left( \epsilon ^{2}e^{2i\left( \pi -t\right) }-1\right)
^{2H+1}dt,
\end{equation*}%
\begin{equation*}
\int_{\mathcal{C}_{2}}f_{\ell }\left( z\right) dz=\frac{1}{2^{2H}i^{2H+1}}%
\int_{0}^{\frac{\pi }{2}}\left( 1+\epsilon e^{i\left( \pi -t\right) }\right)
^{2\ell -2H-1}\left( \epsilon ^{2}e^{2i\left( \pi -t\right) }+2\epsilon
e^{i\left( \pi -t\right) }\right) ^{2H+1}dt.
\end{equation*}%
Careful calculations show that 
\begin{equation}
\func{Im}\int_{\mathcal{L}_{1}}f_{\ell }\left( z\right) dz=0,\
\lim_{\epsilon \rightarrow 0}\left\vert \int_{\mathcal{C}_{j}}f_{\ell
}\left( z\right) dz\right\vert =0,\ j=1,2  \label{eq:Int1}
\end{equation}%
and 
\begin{equation}
\lim_{\epsilon \rightarrow 0}\int_{\mathcal{L}_{2}}f_{\ell }\left( z\right)
dz=\frac{i^{2H+1}}{2^{2H}}B\left( 2H+2,\ell -H+1/2\right)   \label{eq:Int2}
\end{equation}%
Thus, by the Cauchy integral theorem in complex analysis, we have 
\begin{equation*}
\begin{split}
\int_{\mathcal{C}}f_{\ell }\left( z\right) dz& =-\lim_{\epsilon \rightarrow
0}\int_{\mathcal{L}_{1}+\mathcal{C}_{1}+\mathcal{L}_{2}+\mathcal{C}%
_{2}}f_{\ell }\left( z\right) dz \\
& =-\lim_{\epsilon \rightarrow 0}\int_{\mathcal{L}_{1}}f_{\ell }\left(
z\right) dz-\frac{i^{2H+1}}{2^{2H}}B\left( 2H+2,\ell -H+\frac{1}{2}\right) 
\end{split}%
\end{equation*}%
which leads to that 
\begin{equation}
\func{Im}\int_{\mathcal{C}}f_{\ell }\left( z\right) dz=2^{-2H-1}B\left(
2H+2,\ell -H+\frac{1}{2}\right) \sin \left[ \Big(H+\frac{1}{2}\Big)\pi %
\right] .  \label{eq:Imf}
\end{equation}%
in view of the equalities (\ref{eq:Int1}) and (\ref{eq:Int2}). Therefore, by
combining the equalities (\ref{eq:Imf}) and (\ref{d-tilda}), (\ref%
{dti-Imintf}) above, we obtain that%
\begin{equation*}
\widetilde{d}_{\ell }=\frac{2}{\pi }B\left( H+1,\frac{1}{2}\right) B\left(
2H+2,\ell -H+\frac{1}{2}\right) \sin \left[ \Big(H+\frac{1}{2}\Big)\pi %
\right] .
\end{equation*}%
Here, recall the formula 
\begin{equation*}
B\left( a,b\right) =\frac{\Gamma \left( a\right) \Gamma \left( b\right) }{%
\Gamma \left( a+b\right) },a>0,b>0
\end{equation*}%
with Gamma function $\Gamma \left( y\right) =\left( y-1\right) \Gamma \left(
y-1\right) $ for any $y>1$ and the stirling's approximation 
\begin{equation*}
\Gamma \left( y\right) =\sqrt{2\pi y}\left( \frac{y}{e}\right) ^{y}\left(
1+O\left( \frac{1}{y}\right) \right) ,\ \ \hbox{ as }\ y\rightarrow \infty ,
\end{equation*}%
where the error term $O\left( \frac{1}{y}\right) $ being the same order as $%
\frac{1}{y},$ we derive the following estimates: there exists a uniform
constant $C_{1}$ such that for any $\ell \geq 1,$ 
\begin{equation*}
C_{1}^{-1}\ell ^{-\left( 2H+2\right) }\leq \widetilde{d}_{\ell }\leq
C_{1}\ell ^{-\left( 2H+2\right) }.
\end{equation*}%
Let $K_{1}=2C_{1},$ then the inequalities in (\ref{Bound-dl}) are derived in
view of (\ref{eq:d-dti}).
\end{proof}

\section{Strong local nondeterminism}

In order to prove the strong local nondeterminism property of $\left\{
B\left( x\right) ,x\in \mathbb{S}^{2}\right\},$ we first recall the
following lemma, which is a consequence of Proposition 7 in \cite%
{LanMarXiao1}.

\begin{lemma}
\label{Ing} Assume a sequence $\left\{ d_{\ell },\ell =0,1,...\right\} $
satisfies the condition (\ref{Bound-dl}). Then there exists a constant $%
C_{2}>0$ depending on $H$ only, such that for all choices of $n\in \mathbb{N}%
,$ all $x,x_{1},...,x_{n}\in \mathbb{S}^{2},$ and $\gamma _{j}\in \mathbb{R}%
, $ $j=1,2,...,n$, we have 
\begin{equation}
\sum_{\ell =0}^{\infty }\sum_{m=-\ell }^{\ell }d_{\ell }\bigg[Y_{\ell
m}(x)-\sum_{j=1}^{n}\gamma _{j}Y_{\ell m}(x_{j})\bigg]^{2}\geq
C_{2}\varepsilon^{2H},  \label{Eq:YlmControl}
\end{equation}%
where $\varepsilon =\min \left\{ d_{\mathbb{S}^{2}}\left( x,x_{k}\right) ,\
k=1,...,n\right\} $.
\end{lemma}

Now we are ready to state and prove the following theorem. Its conclusion is
referred to as the property of strong local nondeterminism of SFBM $\left\{
B\left( x\right) ,x\in \mathbb{S}^{2}\right\}$.

\begin{theorem}
\label{Th:SLND} For a SFBM $\left\{ B\left( x\right) ,x\in \mathbb{S}%
^{2}\right\} $, there exists a constant $K_{2}>0$ depending only on the
Hurst index $H\in (0,1/2],$ such that for all integers $n\geq 1$ and all $%
x,x_{1},...,x_{n}\in \mathbb{S}^{2},$ we have%
\begin{equation}  \label{Eq:SLND}
\mathrm{Var}\left( B\left( x\right) |B\left( x_{1}\right) ,...,B\left(
x_{n}\right) \right) \geq K_{2}\min_{0\leq k\leq n} d_{\mathbb{S}^{2}}\left(
x,x_{k}\right) ^{2H},
\end{equation}
where $\mathrm{Var}\left( B\left( x\right) | B\left( x_{1}\right)
,...,B\left( x_{n}\right)\right)$ denotes the conditional variance of $B( x)$
given $B\left( x_{1}\right) ,...,B( x_{n})$, and $x_{0}=N$.
\end{theorem}

\begin{remark}
Note that the strong local nondeterminism \eqref {Eq:SLND} here is slightly
different from the SLND property proved in \cite{LanMarXiao1} for isotropic
Gaussian random fields: the minimum on the right hand side of \eqref
{Eq:SLND} is not only taken over $x_{1}, ... ,x_{n}$ but also over $x_{0}=N$%
. This is because of the assumption $B\left( N\right) =0$ in the definition
of SFBM. From statistics viewpoint, $\mathrm{Var}\left( B\left( x\right)
|B\left( x_{1}\right) ,...,B\left( x_{n}\right)\right)$ is the squared error
of predicting the value $B(x)$, given observations of $B$ at locations $x_1,
\ldots, x_n$. Since we already know that $B\left( N\right) =0$, this
information may reduce the prediction error. 
\end{remark}

\textbf{Proof of Theorem \ref{Th:SLND}.} It is known that for a Gaussian
random field $B$,%
\begin{equation*}
\mathrm{Var}\left( B(x)|B\left( x_{1}\right) ,...,B\left( x_{n}\right)
\right) =\inf \left\{ \mathbb{E}\bigg[\Big(B(x)-\dsum\limits_{j=1}^{n}\gamma
_{j}B(x_{j})\Big)^{2}\bigg]\right\} ,
\end{equation*}%
where the infimum is taken over all $\left( \gamma _{1},...,\gamma
_{n}\right) \in \mathbb{R}^{n}.$ Hence, in order to establish \eqref{Eq:SLND}%
, it is sufficient to prove that there exists a positive constant $C_{3}$
such that 
\begin{equation}
\mathbb{E}\bigg[\Big(B(x)-\dsum\limits_{j=1}^{n}\gamma _{j}B(x_{j})\Big)^{2}%
\bigg]\geq C_{3}\,\varepsilon ^{2H}  \label{Eq:SLND2}
\end{equation}%
holds for all $\gamma _{1},...,\gamma _{n}\in \mathbb{R}$, where $%
\varepsilon =\min \left\{ d_{\mathbb{S}^{2}}\left( x,x_{k}\right) ,\
k=0,...,n\right\} .$ Let $\gamma _{0}=1-\dsum\limits_{j=1}^{n}\gamma _{j},$
then it follows from the Karhunen-Lo\`{e}ve expansion (\ref{repBM}) of $%
B\left( x\right) $ and the properties of random coefficients $\{\varepsilon
_{\ell m}\}_{lm}$ in (\ref{indep-coef}) that the left hand side of (\ref%
{Eq:SLND2}) is equal to 
\begin{equation}
\begin{split}
& \mathbb{E}\Bigg\{\bigg[\dsum\limits_{\ell =0}^{\infty
}\dsum\limits_{m=-\ell }^{\ell }i\sqrt{\pi d_{\ell }}\,\varepsilon _{\ell m}%
\Big(Y_{\ell m}\left( x\right) -\dsum\limits_{j=0}^{n}\gamma _{j}Y_{\ell
m}\left( x_{j}\right) \Big)\bigg]^{2}\Bigg\} \\
& =\dsum\limits_{\ell =0}^{\infty }\dsum\limits_{m=-\ell }^{\ell }\pi
d_{\ell }\mathbb{E}\big(\left\vert \varepsilon _{\ell m}\right\vert ^{2}\big)%
\bigg \vert Y_{\ell m}\left( x\right) -\dsum\limits_{j=0}^{n}\gamma
_{j}Y_{\ell m}\left( x_{j}\right) \bigg \vert^{2} \\
& =\pi \dsum\limits_{\ell =0}^{\infty }\dsum\limits_{m=-\ell }^{\ell
}d_{\ell }\bigg \vert Y_{\ell m}\left( x\right)
-\dsum\limits_{j=0}^{n}\gamma _{j}Y_{\ell m}\left( x_{j}\right) \bigg \vert%
^{2},
\end{split}
\label{Eq2}
\end{equation}%
where $x_{0}=N.$ Therefore (\ref{Eq:SLND2}) is an immediate consequence of
Lemma \ref{Ing} in view of the equalities (\ref{Eq2}). This completes the
proof of (\ref{Eq:SLND}). {\ \rule{0.5em}{0.5em}} 


\end{document}